\documentclass[english,a4paper]{amsart}
\usepackage[scaled=0.86]{helvet}
\usepackage[T1]{fontenc}
\usepackage[latin9]{inputenc}
\usepackage{amsthm}
 \usepackage{graphicx}
\usepackage{amssymb}
 
\makeatletter
 \theoremstyle{plain}
\newtheorem{thm}{Theorem}[section]
  \theoremstyle{plain}
  \newtheorem{prop}[thm]{Proposition}
  \theoremstyle{plain}
  \newtheorem{lem}[thm]{Lemma}
 \theoremstyle{definition}
 \newtheorem*{defn*}{Definition}
  \theoremstyle{plain}
  \newtheorem{cor}[thm]{Corollary}
  \theoremstyle{remark}
  \newtheorem{rem}[thm]{Remark}

\usepackage{amsthm}

\usepackage[all]{xy}

\usepackage{psfrag}

\usepackage{bbm}

\DeclareMathOperator{\diam}{diam}

\DeclareMathOperator{\card}{card}

\newcommand{\Z}{\mathbb{Z}}

\newcommand{\R}{\mathbb {R}}

\newcommand{\N}{\mathbb {N}}

\renewcommand{\1}{\mathbbm{1}}

\renewcommand{\d}{\mathrm {d}}

\renewcommand{\hat}{\widehat}
\renewcommand{\phi}{\varphi}
\renewcommand{\epsilon}{\varepsilon}
\DeclareMathOperator*{\starlim}{\ast-lim}

\usepackage{babel}
 \makeatother

\usepackage{babel}

\begin{document}
\selectlanguage{english}

\title[A dichotomy between uniform distributions]{A dichotomy between uniform distributions of the  Stern-Brocot 
and  the  Farey sequence}

\author{Marc Kesseböhmer}

\address{Fachbereich 3 -- Mathematik und Informatik, Universität Bremen, Bibliothekstr.
1, D--28359 Bremen, Germany}

\email{mhk@math.uni-bremen.de}

\author{Bernd O. Stratmann}

\address{Fachbereich 3 -- Mathematik und Informatik, Universität Bremen, Bibliothekstr.
1, D--28359 Bremen, Germany}

\email{bos@math.uni-bremen.de}

\date{\today}
\begin{abstract}
We employ infinite ergodic theory to show that the
even Stern-Brocot sequence and the Farey sequence are  uniformly distributed mod
1 with respect to certain canonical weightings. As a corollary we derive
the precise asymptotic for the Lebesgue measure of continued fraction
sum-level sets as well as connections to asymptotic behaviours  of geometrically
and arithmetically restricted Poincaré series. Moreover, we give relations of our main results 
 to elementary observations for the Stern-Brocot tree.
\end{abstract}



\maketitle

\section{Introduction and statements of result}

 In this paper we consider weighted uniform distributions (mod 1)  for the following two canonical sequences:
the {\em Farey sequence} $\left(\mathcal{F}_{n}\right)_{n \in \N}$ which is given
by \[
\mathcal{F}_{n}:=\left\{ p/q:0<p\leq q\leq n, \mbox{gcd}(p,q)=1\right\}, \]
 and the \emph{even Stern-Brocot sequence}  $\left(\mathcal{S}_{n}\right)_{n \in \N}$ which is given by \[
\mathcal{S}_{n}:=\left\{ s_{n,2k}/t_{n,2k}:k=1,\ldots,2^{n-1}\right\} , \]
where  the integers $s_{n,k}$ and $t_{n,k}$
 are defined recursively by
\begin{itemize}
\item[] $s_{0,1}:=0\, \,$ and $\, \, s_{0,2}:=t_{0,1}:=t_{0,2}:=1$; 
\item[] $s_{n+1,2k-1}:=s_{n,k}\quad\textrm{and}\quad t_{n+1,2k-1}:=t_{n,k},$
for $k=1,\ldots,2^{n}+1$; 
\item[] $s_{n+1,2k}:=s_{n,k}+s_{n,k+1}\quad\textrm{and}\quad t_{n+1,2k}:=t_{n,k}+t_{n,k+1}$,
for $k=1,\ldots2^{n}$. 
\end{itemize}
\noindent The following theorem states the main results of this paper, where $\delta_{x}$ denotes the Dirac distribution at $x \in [0,1]$, $\starlim$
the weak limit of measures, and $\lambda$  the Lebesgue
measure on $\left[0,1\right]$. Note that, throughout, all appearing fractions will always be assumed to be reduced. 
\begin{thm} 
\label{Thm:atomic2} 
For the  even Stern-Brocot sequence we have that 
\begin{eqnarray}
\label{Stern-Brocot}
\starlim_{n \to \infty} \quad  \log (n^2)\sum_{p/q\in\mathcal{S}_{n}}\,
q^{-2}\,\delta_{p/q} =\lambda,\end{eqnarray}
and for the Farey sequence we have that 
\begin{equation}
\label{Farey}
\starlim_{n \to \infty}\quad \frac{\zeta(2)}{\log n}
\sum_{p/q\in\mathcal{F}_{n}}\,q^{-2}\,\delta_{p/q}= 
\lambda. 
\end{equation}
\end{thm}
In fact, for the derivation of the assertion in  (\ref{Stern-Brocot}) we will show that the following more general measure theoretical result
holds. In here, $T:[0,1] \to [0,1]$ denotes the {\em Farey map} defined by \[
T\left(x\right):=\left\{ \begin{array}{lll}
x/(1-x) & \,\,\textrm{for} & x\in\left[0,1/2\right]\\
(1-x)/x & \,\,\textrm{for} & x\in\left(1/2,1\right].\end{array}\right.\]
\begin{thm} 
\label{Thm:interval} 
For each  rational number $v/w\in\left(0,1\right]$ we have that \begin{equation}
 \starlim_{n \to \infty} \quad  \log (n^{vw})  \sum_{p/q\in T^{-n}\left\{ v/w\right\} }\, q^{-2}\, \delta_{p/q}=\lambda. \label{eq:Stronger}\end{equation}
\end{thm}

In a nutshell,  the proofs of these results are obtained as follows.  The convergence in (\ref{Farey}) is derived from combining {\em Toeplitz's Lemma} and a classical result by Landau \cite{Landau} and Mikol\'as \cite{Mikolas} with a well-know estimate for the {\em Euler totient function}     $\varphi(n) := \; \card \{ 1 \le m \le n : \mbox{gcd}(m,n) = 1 \} $.  Whereas,  the proof of Theorem \ref{Thm:interval},
and   consequently the proof of (\ref{Stern-Brocot}), is obtained from the following slightly more technical
 result, which will be derived by employing some  recent progress
  in  infinite ergodic theory. 
 
 \begin{prop}
\label{thm:local} For each interval $\left[\alpha,\beta\right] \subset (0,1]$ 
we have that
\[ \starlim_{n \to \infty} \quad \left(\frac{\log n}{\log\left(\beta/\alpha\right)}\cdot \lambda|_{T^{-n}([\alpha,\beta])}\right)= \lambda.\]
\end{prop}

The result in
 Proposition \ref {thm:local} has the following 
immediate elementary number theoretical  implication, which has been  the main result of \cite{KessStrat09} and which there led to the confirmation of a conjecture by Fiala and Kleban \cite{FK} (see also Remark \ref{sumlevel} following the proof of Proposition \ref {thm:local}). In particular, Proposition \ref {thm:local} hence gives rise to an alternative proof of this conjecture.
But let us first 
recall that the regular continued fraction expansion of a number $x \in (0,1]$ is given by \[
x=: [x_{1},x_{2},\ldots]:=\cfrac{1}{x_{1}+\cfrac{1}{x_{2}+\ldots}},\]
 where all the $x_{i}$ are positive integers. Also, we write $a_{n}\sim b_{n}$ if $\lim_{n\to\infty}a_{n}/b_{n}=1$.
\begin{cor}
\label{cor1} We have that
\[
\lambda\left(\left\{[x_{1},x_{2},\ldots]:\sum_{i=1}^{k}x_{i}=n,k\in\N\right\}\right)\sim\frac{1}{\log_{2}n}.\]
\end{cor}
Further immediate consequences of the results in Theorem \ref{Thm:atomic2} and Theorem \ref{Thm:interval}  are given in the following
two corollaries.
\begin{cor} We have that  \[
\starlim_{n \to \infty}  \frac{\zeta(2)}{n}\sum_{{p/q \in [0,1] \atop  2\log q\leq n}}\,q^{-2}\,\delta_{p/q}=
\lambda \, \, \mbox{ and } \, \,  \starlim_{n \to \infty}  \frac{\log (n^2)}{n}\sum_{{p/q=\left[x_{1},\ldots,x_{k}\right]\atop \sum_{i=1}^{k}x_{i}\leq n}}\,q^{-2}\,\delta_{p/q}=\lambda.\]
\end{cor}
The latter dichotomy can also be expressed in more down-to-earth terms as a dichotomy between
{\em  partial geometric Poincaré sums} and  {\em partial algebraic Poincaré
sums}  for the modular group $\Gamma:=PSL_{2}\left(\Z\right)$. For  results of this type on the algebraic growth
rates of Poincaré series for more general Kleinian groups we refer to \cite{KessStrat09b}.
In the following, $d$ refers to the hyperbolic metric in the upper plane model of hyperbolic space  and  $\left| \, \cdot \, \right|$ denotes the word length  in  $\Gamma$ with
respect to the two generators $z\mapsto z+1$ and $z\mapsto-1/z$ of the modular group $\Gamma$.
Also, we write $a_{n}\asymp b_{n}$ if $a_{n}/b_{n}$ is uniformly
bounded away from zero and infinity.
\begin{cor} We have that  \[
\sum_{{\gamma \in \Gamma \atop d\left(0,\gamma\left(0\right)\right)\leq n}}e^{-d\left(0,\gamma\left(0\right)\right)}\asymp n \quad\mbox{and }\sum_{{\gamma \in \Gamma \atop \left|\gamma\right|\leq n}}e^{-d\left(0,\gamma\left(0\right)\right)}\asymp\frac{n}{\log n}.\]
\end{cor}

\begin{rem} ($i$)
Note that the results in Theorem \ref{Thm:atomic2}  complement 
well-known results on  weak convergence 
of empirical measures with constant weight $1$  for the sequences $\left(\mathcal{F}_{n}\right)$ and $\left(\mathcal{S}_{n}\right)$.
More precisely, in \cite{Mikolas} (see also \cite{CodecaPerelli,Dufner,Kopriva1,Kopriva2,Landau})
 it was shown that $\left(\mathcal{F}_{n}\right)$ is uniformly distributed, that is,
\begin{equation}
\label{fac:Farey}
\starlim_{n \to \infty}\, \,   \frac{1}{\card\left(\mathcal{F}_{n}\right)}\sum_{p/q\in\mathcal{F}_{n}}\delta_{p/q}=\lambda.\end{equation}
On the other hand, it is known that the Stern-Brocot sequence is not uniformly distributed.
In fact,  an immediate consequence of the results
in \cite{KesseboehmerStratmann08} is that   \[
\starlim_{n \to \infty}\, \,   \frac{1}{\card\left(\mathcal{S}_{n}\right)}\sum_{p/q\in\mathcal{S}_{n}}\,\delta_{p/q}=m_{T},
\]
where  $m_{T}$ refers to the {\em measure of maximal entropy} for the Farey map $T$. Here,
the reader might like to recall that the distribution function of
$m_{T}$ is equal to the {\em Minkowski question mark function} (see e.g.
\cite{KesseboehmerStratmann08}) and hence, the two measures $m_{T}$ and $\lambda$
are mutually singular. In fact, a numerical calculation has shown that the Hausdorff dimension $\dim_{H}(m_T):=\inf\left\{ \dim_{H}(X):m_T(X)=1\right\}$ of the measure $m_T$
is approximable  equal to  $0.875$ (see e.g.  \cite{KesseboehmerStratmann08,L,TU}).

($ii$) In order to tie the results in Theorem \ref{Thm:atomic2} (1) and Theorem \ref{Thm:interval} to  elementary number theory
and,  in particular, to give a clarification of the factor $vw$ in Theorem \ref{Thm:interval}, we mention the following observation for the 
even Stern-Brocot tree. For each reduced fraction $v/w \in (0,1)$ and for all $n \in \N_0$, we have
\begin{equation}
 \label{fac:Stern2}  \sum_{p/q \in T^{-n}(v/w)} \frac{1}{pq} = \frac{1}{vw}.
 \end{equation}
To see this first in an elementary way, note that we have  
$ p/q \in \mathcal{S}_n $ if and only if $ T^{-1}(p/q) = \{p/(p+q), q/(p+q)\} \subset \mathcal{S}_{n+1}$. Furthermore, 
with $\kappa:  \bigcup_{n \in \N} \mathcal{S}_n \to \R$ given by  $\kappa( p/q):= 1/(pq)$, one immediately verifies that
\[ \kappa( p/(p+q)) + \kappa(q/(p+q)) = \kappa(p/q).\]
The proof now follows by induction. 
Note that for the special case $v/w=1/2$ one immediately verifies that  $\mathcal{S}_n=T^{-(n-1)} (1/2)$,
and then  (\ref{fac:Stern2}) becomes
\[
 \sum_{p/q \in \mathcal{S}_n} \frac{2}{pq} = 1, \mbox{ for all $n \in \N$},
\]
which has also been observed by
the  Canadian music theorist 
 Pierre Lamothe (see the reference by Bogomolny in \cite{inter}).
\\
Alternatively, the equality in (\ref{fac:Stern2}) can also be deduced immediately  from the well-known
fixed point equation for the Perron-Frobenius operator $\mathcal{L}$ associated with  the  Farey map $T$ (see Section \ref{sec:IE} for the definition).
For this let $h$ denote the eigenfunction of $\mathcal{L}$ associated with the eigenvalue $1$. It is well known that 
 $h$  is  given by $h(x):=1/x$, which consequently gives that  
 \[ \sum_{y \in T^{-n} (x)} |(T^{n})'(y)|^{-1}  h(y) = h(x), \mbox{for all $x \in (0,1)$ and $n \in \N_0$}. \]
 Since  $|(T^{n})'(p/q)|= q^2/w^2$ for all $p/q  \in T^{-n} (v/w)$, the statement in (\ref{fac:Stern2}) follows. \\
 Finally, let us  apply  Theorem \ref{Thm:interval} 
to obtain yet another proof of the statement in (\ref{fac:Stern2}), and  this proof  will then implicitly use dual aspects of the 
Perron-Frobenius operator. More precisely, by applying Theorem \ref{Thm:interval} twice, we obtain the following, which immediately implies (\ref{fac:Stern2}). For
 each $n \in \N_0$ and for every reduced fraction $v/w \in (0,1)$, we have  
  \begin{eqnarray*}  \sum_{p/q \in T^{-n}(v/w)}  \frac{1}{pq} \cdot \lambda & = &
\sum_{p/q \in T^{-n}(v/w)} \starlim_{k \to \infty}  \,   \log k  \sum_{r/s \in T^{-k} (p/q)} \,
q^{-2}\,\delta_{p/q}\\
\hspace{-4mm} & = & \frac{1}{vw} \starlim_{k \to \infty}   \,   \log (k^{vw}) \sum_{p/q\in T^{-(n+k)}(v/w) }\,
q^{-2}\,\delta_{p/q} = 
\frac{1}{vw} \cdot \lambda.
\end{eqnarray*}
\end{rem}

\section{Proofs of Theorem \ref{Thm:atomic2}, \ref{Thm:interval} and Proposition \ref{thm:local}}

\subsection{Proof of Proposition \ref{thm:local}}\label{sec:IE}
As already mentioned in the introduction, the proof of the Proposition
\ref{thm:local} will make use of some results from infinite
ergodic theory.  
Therefore, let us first recall a few basic facts and results from infinite
ergodic theory for the Farey map. (For an overview, further definitions
and details concerning infinite ergodic theory in general, the reader
is referred to \cite{Aaronson:97}.) It is well known that the 
 {\em Farey system} $\left([0,1],T,\mathcal{A},\mu\right)$
is a conservative ergodic measure preserving dynamical systems. Here,
$\mathcal{A}$ refers to the Borel $\sigma$-algebra of $[0,1]$
and the measure $\mu$ is the infinite $\sigma$-finite $T$-invariant
measure absolutely continuous with respect to the Lebesgue measure
$\lambda$. (Recall that {\em conservative and ergodic} means that for
all $f\in L_{1}^{+}\left(\mu\right):=\left\{ f\in L_{1}\left(\mu\right):\; f\geq0\;\mathrm{and}\;\mu(f\cdot\1_{[0,1]})>0\right\} $
we have $\mu$--almost everywhere $\sum_{n\geq0}\hat{T}^{n}\left(f\right)=\infty$,
where $\1_{[0,1]}$ refers to the characteristic function
of $[0,1]$; also, {\em invariance of $\mu$ under $T$} means
$\hat{T}\left(\1_{[0,1]}\right)=\1_{[0,1]}$,
where $\hat{T}$ denotes the transfer operator defined below.) In
fact, with $\phi_{0}:[0,1]\to[0,1]$ defined by
$\phi_{0}(x):=x$, the measure $\mu$ is explicitly given by \[
\mathrm{d}\lambda=\phi_{0}\,\mathrm{d}\mu.\]
 Moreover, recall that the {\em transfer operator} $\hat{T}:L_{1}\left(\mu\right)\to L_{1}\left(\mu\right)$
associated with the Farey system is the positive
linear operator which is given by \[
\mu\left(\1_{C}\cdot\hat{T}\left(f\right)\right)=\mu\left(\1_{T^{-1}\left(C\right)}\cdot f\right),\mbox{ for all }f\in L_{1}\left(\mu\right),C\in\mathcal{A}.\]
 Finally, note that the {\em Perron-Frobenius operator} $\mathcal{L}:L_{1}\left(\mu\right)\to L_{1}\left(\mu\right)$
of the Farey system is given by \[
\mathcal{L}\left(f\right)=\left|u_{0}'\right|\cdot(f\circ u_{0})+\left|u_{1}'\right|\cdot(f\circ u_{1}),\mbox{ for all }f\in L_{1}\left(\mu\right),\]
where $u_0$ and $u_1$ refer to the inverse branches of $T$, which are given for $x \in [0,1]$  by
\[ u_{0}\left(x\right)=x/(1+x)\mbox{ and }u_{1}\left(x\right)=
 1/(1+x).\]
 One then immediately verifies that the two operators $\hat{T}$ and
$\mathcal{L}$ are related through \[
\hat{T}\left(f\right)=\phi_{0}\cdot\mathcal{L}\left(f/\phi_{0}\right),\mbox{ for all }f\in L_{1}\left(\mu\right).\]
 Now, the crucial notion for proving Proposition \ref{thm:local} is provided by the following concept of a uniformly
returning set which was introduced in \cite{KesseboehmerSlassi:05}. \begin{itemize}
\item [] 
 \textit{A set $C\in\mathcal{A}$
with $0<\mu\left(C\right)<\infty$ is called {\em uniformly returning  for}
$f\in L_{\mu}^{+}$ if there exists a positive increasing sequence
$\left(w_{n}\right)_{n \in \N}$ of positive reals such
that $\mu$--almost everywhere and uniformly in $C$ we have \[
\lim_{n\to\infty}w_{n}\hat{T}^{n}\left(f\right)=\mu(f).\]
}  
\end{itemize}
In \cite{KesseboehmerSlassi:05}[Lemma 3.3] it was shown that for the Farey system we have that every interval contained 
in $[1/2,1]$ is uniformly returning, for each function $f$ which has the property that
\[ \hat{T}^{n}\left(f\right) \in \mathcal{D}:= \left\{ g\in C^{2}\left(\left[0,1\right]\right):g'\geq0,g''\leq0\right\} .\]
Moreover, in \cite[Section 3.1]{KesseboehmerSlassi:05} it was shown  that in the situation of the Farey system the sequence $\left(w_{n}\right)_{n \in \N}$ can
 be chosen to be equal to $\left(\log n\right)_{n \in \N}$. (For further examples of one dimensional dynamical
systems which allow uniformly returning sets for some appropriate
functions we refer to \cite{Thaler:00}.)  We are now in the position to give the proof of Proposition
\ref{thm:local}.
 \begin{proof}
[Proof of Proposition \ref{thm:local} ] Consider the function  $\phi_{t}$ given by
$\phi_{t}:x\mapsto x\cdot\exp\left(tx\right)$. The first aim is to show that for all $t\in\left[-1,1\right]$ we have \[ \widehat{T}\phi_{t}
\in \mathcal{D} .
\]
Indeed, for  $t \in [-1, 0]$ this is an immediate consequence of the facts that 
$\phi_{t}$ is increasing and concave and that $\widehat{T}\left(\mathcal{D}\right)\subset\mathcal{D}$.
For $t\in(0,1]$,
a straight forward computation shows that the first
derivative at $x \in [0,1]$ is given by \[
\left(\widehat{T}\phi_{t}\right)'(x)=\frac{\phi_{t}'\left(\frac{x}{x+1}\right)-x\phi_{t}'\left(\frac{1}{x+1}\right)}{\left(x+1\right)^{3}}+\frac{\phi_{t}\left(\frac{1}{x+1}\right)-\phi_{t}\left(\frac{x}{x+1}\right)}{\left(x+1\right)^{2}}.\]
 For the second derivative we then obtain \begin{eqnarray*}
\left(\widehat{T}\phi_{t}\right)''\left(x\right) & = & \frac{\left(-2\, xt-6x+2t+xt^{2}+2x^{3}-4tx^{2}-4\right)\exp\left(\frac{tx}{x+1}\right)}{\left(x+1\right)^{6}}\\
 &  & \:+\frac{\left(2tx-6x-2t+xt^{2}+2x^{3}+4tx^{2}-4\right)\,\exp\left(\frac{t}{x+1}\right)}{\left(x+1\right)^{6}}.\end{eqnarray*}
 This immediately implies that $\left(\widehat{T}\phi_{t}\right)''\leq0$,
for all $t\in(0,1]$. Therefore, $\left(\widehat{T}\phi_{t}\right)'$
is decreasing on $[0,1]$ with $\left(\widehat{T}\phi_{t}\right)'\left(1\right)=0$,
which shows that on $[0,1]$ we have that $\left(\widehat{T}\phi_{t}\right)'\geq 0$.
Hence, we can apply  \cite[Lemma 3.2]{KesseboehmerSlassi:08}, which then implies that
$\widehat{T}\phi_{t}\in\mathcal{D}$,
for all $t\in\left[-1,1\right]$.\\
We proceed by noting that \cite[Lemma 3.3]{KesseboehmerSlassi:08} guarantees  that every interval contained in 
 $[1/2,1]$  is a uniformly returning
 set for $\phi_t$, for each $t \in [-1,1]$. 
In order to complete the proof of the proposition, we employ the method of moments as follows. 
 The aim is to show that for each $[\alpha, \beta] \subset (0,1]$ and for  each $t\in[-1,1]$ we have for the moment generating
function at $t$  that  \[
\lim_{n\to\infty}\int\exp\left(tx\right)\cdot\frac{\log n}{\mu\left([\alpha, \beta]\right)}\cdot\1_{T^{-n}\left([\alpha, \beta]\right)}\d\lambda(x)=\int\exp\left(tx\right)\d\lambda(x).\]
To see this, we argue by induction as follows. For $[\alpha, \beta] \subset [1/2,1]$, we have that
\begin{eqnarray*}
 &  & \hspace{-1cm}\lim_{n\to\infty}\int\exp\left(tx\right)\cdot\frac{\log n }{\mu\left([\alpha, \beta]\right)}\cdot\1_{T^{-n}\left([\alpha, \beta]\right)}\left(x\right)\d\lambda(x)\\
 & = & \lim_{n\to\infty} \frac{\log n}{\mu\left([\alpha, \beta]\right)}\cdot \mu\left(\phi_{t}\cdot \1_{T^{-n}\left([\alpha, \beta]\right)}\right)=\lim_{n\to\infty}\frac{\log n}{\mu\left([\alpha, \beta]\right)}\cdot\mu\left(\widehat{T}^{n}\phi_{t}\cdot\1_{[\alpha, \beta]}\right)\\ &=&\mu\left(\phi_{t}\right)
  =  \int\exp\left(tx\right)\,\d\lambda(x).\end{eqnarray*}
Next, suppose that the assertion holds for any interval which is contained in the set $\mathcal{E}_n:=\bigcup_{k=0}^{n-1} T^{-k} ([1/2,1])$,
and consider an interval  $[\alpha, \beta] \subset T^{-n} ([1/2,1]) \setminus \mathcal{E}_n$. Since $T([\alpha,\beta]) \subset
 \mathcal{E}_n$, we then have
\begin{eqnarray*}
 &  & \hspace{-1cm}\lim_{m\to\infty}\int\exp\left(tx\right)\cdot\frac{\log m }{\mu\left([\alpha, \beta]\right)}\cdot\1_{T^{-m}\left([\alpha, \beta]\right)}\left(x\right)\d\lambda(x)\\
 & = & \lim_{m\to\infty}\frac{\log m }{\mu\left([\alpha, \beta]\right)}\cdot\mu\left(\widehat{T}^{m}\phi_{t}\cdot \left(\1_{
 T^{-1}(T([\alpha, \beta]))} - \1_{
 T^{-1}(T([\alpha, \beta]))\cap \mathcal{E}_n} \right)\right)\\ &=&
\lim_{m\to\infty}\frac{\log m }{\mu\left([\alpha, \beta]\right)}\cdot \left(\mu\left(\widehat{T}^{m+1}\phi_{t}\cdot \1_{
 T([\alpha, \beta])} \right) - \mu\left(\widehat{T}^{m}\phi_{t}\cdot \1_{
 T^{-1}(T([\alpha, \beta]))\cap \mathcal{E}_n} \right) \right)  \\ &=& 
  \frac{ \mu\left(\phi_{t}\right)}{\mu\left([\alpha, \beta]\right)} \left( \mu(T([\alpha, \beta]))- \mu(T^{-1}(T([\alpha, \beta]))\cap \mathcal{E}_n) \right)
  =  \int\exp\left(tx\right)\,\d\lambda(x).\end{eqnarray*}
 This finishes the proof of Proposition \ref{thm:local}. 
\end{proof}

\begin{rem}
\label{sumlevel}
We  remark that the assertion in Corollary \ref{Thm:interval} is an immediate consequence of Proposition \ref{thm:local}.  
Indeed, by choosing $[\alpha,\beta] = [1/2,1]$ and observing that  (see \cite[Lemma 2.1]{KessStrat09})
\[      T^{-(n-1)} \left([1/2,1]\right) =  \left\{[x_{1},x_{2},\ldots]:\sum_{i=1}^{k}x_{i}=n,k\in\N\right\},\]
it follows that
\[
\lambda\left(\left\{[x_{1},x_{2},\ldots]:\sum_{i=1}^{k}x_{i}=n,k\in\N\right\}\right)\sim\frac{\log 2}{\log n}.\]
\end{rem}

\subsection{Proof of Theorem  \ref{Thm:interval}}

The following two lemmata will be required in the proof of  Theorem  \ref{Thm:interval}. 
Note that the first lemma of these has already been obtained in  \cite{KessStrat09}. However, in order
to keep the paper as self contained as possible, we include a proof here.
\begin{lem}
\textup{\label{cor:GrowthRateBound} \[
  \sum_{p/q\in\mathcal{S}_{n}}q^{-2}\asymp \frac{1}{\log n}.\]
}\end{lem}
\begin{proof}
First note that
there is a 1--1 correspondence between the 
sequence $\left(\mathcal{S}_{n}\right)$ and the set of connected components of  $T^{-(n-1)}\left([1/2,1]\right)$. That is,  if $p/q=[a_{1},\ldots,a_{n}]\in \mathcal{S}_{n}$, where 
$a_{n}>1$, then one of these connected component is given by \begin{eqnarray*}
C_{n}\left(p/q\right) & := & \{[x_{1},x_{2},\ldots]:x_{i}=a_{i}\:\mbox{for }1\leq i\leq n\}\:\\
 &  & \qquad \cup \{[x_{1},x_{2},\ldots]:x_{i}=a_{i}\:\mbox{for }1\leq i\leq n-1, x_{n}=a_{n}-1,x_{n+1}=1\}.
 \end{eqnarray*}
Using standard Diophantine estimates we find that $\lambda\left(C_{n}\left(p/q\right)\right)\asymp1/q^{2}$.
Hence, an application of Corollary \ref{cor1} finishes the proof of the lemma. 
\end{proof}
For the
next lemma note that the sequence $\left(\mathcal{S}_{n}\right)$
can also be expressed in terms of the inverse branches $u_1$ and $u_2$ of the Farey map $T$.  
Namely, one immediately verifies that the orbit of the unit interval under the free
semi-group $\Phi$ generated by $u_1$ and $u_2$ is in 1--1 correspondence
to the set of all Stern-Brocot intervals
\[ \left\{
\left[\frac{s_{n,k}}{t_{n,k}},\frac{s_{n,k+1}}{t_{n,k+1}}\right):\,
n \in \N_0; k=1,\ldots,2^{n}\right\} .\] 
 Note that for each rational number $v/w \in (0,1]$ we have that \[
\{T^{-n}\left\{ v/w\right\}: n\in \N\} =\left\{ \gamma \left(v/w \right): \gamma \in \Phi \right\} .\]
Moreover, note that the $\Phi$-orbit of $1$ is equal to the set of rational numbers in $(0,1)$.
More precisely, we have that if  $\gamma \in \Phi$ then $\gamma(1)=v/w$, for some $v,w \in \N$ such that $v<w$ and $\mbox{gcd} (v,w)=1$, and for the  modulus of the  derivative of $\gamma$ at $1$ we then have that  $|\gamma'(1)|=w^{-2}$.

In the following we let $\mathcal{U}_\epsilon (x)$ denote the interval centred at $x \in \R$ of Euclidean diameter $\diam (\mathcal{U}_\epsilon (x))$ equal to $\epsilon>0$.

\begin{lem}
\label{diam}
For each $ g \in \Phi$ there exists $\Delta: (0,1] \to \R_+$
with $\lim_{s \to 0}\Delta(s)=0$ such that for each  $ h \in \Phi$ and 
 $\epsilon >0$ sufficiently small,  we have  \[
 \left| \diam(h(\mathcal{U}_\epsilon (g(1))))-  \epsilon \, |(h'(g(1))|\right| < \epsilon \, |(hg)'(1)| \, \, \Delta(\epsilon).\]
 \end{lem}

\begin{proof}
By the bounded distortion property, we have  for each $ z \in (0,1)$ that there exists 
$\Delta_z: (0,1] \to \R_+$
with $\lim_{s \to 0}\Delta_z(s)=0$ such that, for each $\epsilon> 0$ sufficiently small,
\[ \sup_{x,y \in \mathcal{U}_\epsilon (z) \atop \gamma  \in \Phi} \left| \frac{|\gamma'(x)|}{|\gamma'(y)|} -1 \right| < 
\Delta_z(\epsilon).\]
This implies that for fixed $g\in \Phi$  we have, for each $h \in \Phi$ and  $\epsilon> 0$ sufficiently small,
 \[ \left| \frac{\diam (h(\mathcal{U}_{\epsilon}(g(1))))}{\epsilon |h'(g(1))|} -1 \right| < 
\Delta_{g(1)} (\epsilon).\]
From this we deduce that
\begin{eqnarray*} \left|\diam (h(\mathcal{U}_{\epsilon}(g(1)))) -\epsilon  |h'(g(1))| \right| &< & 
\epsilon \, \frac{|(hg)'(1))|}{|g'(1)|} \,   \Delta_{g(1)} (\epsilon)=\epsilon  |(hg)'(1)| \, \,  \Delta(\epsilon).\\
\end{eqnarray*}
This finishes the proof.
\end{proof}

\begin{proof}[Proof of Theorem  \ref{Thm:interval}]
Let $g \in \Phi$ be given and define, for $\epsilon >0$ sufficiently small,
\[ \mathcal{U}_{g,\epsilon,n}:=
T^{-(n-1)}\left(\mathcal{U}_\epsilon(g(1))\right).\]
Let $u_{g,\epsilon}:=1 / \mu(\mathcal{U}_\epsilon(g(1)))=
1 / \log \left((g(1)+ \epsilon  /2)/(g(1)-\epsilon /2)\right)$, and consider the measure 
$\nu_{g,\epsilon,n}$ which is given, for each $n\in\N$, by 
  \[ 
\nu_{g,\epsilon,n}= u_{g,\epsilon}\log n \cdot \lambda|_{\mathcal{U}_{g,\epsilon,n}}.\]
By Proposition \ref{thm:local}, we then have that $\starlim_{n \to \infty} \nu_{g,\epsilon,n}=\lambda$.
Also,    consider the atomic measure $\rho_{g,\epsilon,n}$ which
 is given, for each $n\in\N$, by 
\[ \rho_{g,\epsilon,n}:= u_{g,\epsilon}\log n     \sum_{f(1)\in T^{-(n-1)}(g(1))}   \epsilon \,  \frac{|f'(1)|}{|g'(1)|}  \, \cdot \delta_{f(1)}.\]
Now, observe that 
\[   \lim_{\epsilon \searrow 0}  {\epsilon u_{g,\epsilon}}=  \lim_{\epsilon \searrow 0} \frac{\epsilon}{ \log \frac{g(1)+\epsilon /2}{g(1)- \epsilon /2}}=  \lim_{\epsilon \searrow 0} \frac{\epsilon}{\epsilon/ ( g(1)-
\epsilon/2)} = g(1), \]
and let the measures $ \rho_{g,n}$ be defined by
\[ \rho_{g,n}:= g(1) \log n     \sum_{f(1)\in T^{-(n-1)}(g(1))}  \,  \frac{|f'(1)|}{|g'(1)|}  \, \cdot \delta_{f(1)}.\] 
Using  Lemma \ref{cor:GrowthRateBound} and Lemma \ref{diam}, we now obtain the following for all $x \in [0,1]$, where  
$F^{(\nu)}_{g,\epsilon,n}$,  $F^{(\rho)}_{g,\epsilon,n}$ and $F^{(\rho)}_{g,n}$ denote the distribution functions  
of the measures $\nu_{g,\epsilon,n}$,
 $ \rho_{g,\epsilon,n} $ and $\rho_{g,n}$, and where we write $a_{n} \ll b_{n}$ if $a_{n}/b_{n}$ is uniformly
bounded  from above,
\begin{eqnarray*}
 \left| F^{(\nu)}_{g, \epsilon,n}(x) -F^{(\rho)}_{g,n} (x)\right| &\leq &
\left| F^{(\nu)}_{g,\epsilon,n}(x) -F^{(\rho)}_{g,\epsilon,n} (x)\right| 
+ \left| F^{(\rho)}_{g,\epsilon,n}(x) -F^{(\rho)}_{g,n} (x)\right|  \\ 
 &\hspace{-1.6cm}\ll& \hspace{-4mm}  u_{g,\epsilon}  \log n \hspace{-5mm} \sum_{hg(1)\in T^{-(n-1)}(g(1))}  \left| \diam (h(\mathcal{U}_\epsilon(g(1))) -  \epsilon  \,  \frac{|(hg)'(1)|}{|g'(1)|} \right|\\
 &\hspace{-1.4cm}& \hspace{-5mm}  +\frac{\epsilon u_{g,\epsilon} \log n}{n^{2}}+ \,\, |g(1)-  \epsilon u_{g,\epsilon}| \log n  \sum_{f(1)\in T^{-(n-1)}(g(1))}   |f'(1)|  \\
 &\hspace{-1.3cm}\ll& \hspace{-6mm}  \left(\epsilon \,  u_{g,\epsilon}   \, \Delta(\epsilon)   +
   |g(1)-  \epsilon u_{g,\epsilon}| \right)  \, \,  \log n  \, \sum_{f(1)\in T^{-(n-1)}(g(1))}   |f'(1)| \\
 &\hspace{-1.5cm}\ll& \hspace{-6mm}    |g(1)-  \epsilon u_{g,\epsilon}|+   g(1)  \, \Delta(\epsilon).
\end{eqnarray*}
This holds for   $\epsilon>0$ arbitrary small and hence, we obtain  that  
\[ \starlim_{n \to \infty}  \rho_{g,n} = \lambda.\]
The proof of Theorem \ref{Thm:interval} now follows, if we insert in the definition of $\rho_{g,n}$ the fact that  $g(1)$ can be written in form of a reduced fraction $v/w$ and  that then $|g'(1)|=w^{-2}$, as well as similarly, that $f(1)$ can be written in form of a reduced fraction $p/q$ and  that then $|f'(1)|=q^{-2}$.
\end{proof}

\subsection{Proof  of Theorem  \ref{Thm:atomic2} (\ref{Farey})}
\begin{proof}
  Define $\mathcal{F}_{n}^{*}:=\left\{ p/n:0<p\leq n, \mbox{gcd}\left(p,n\right)=1\right\} $
and $\psi\left(n\right):=\card\left(\mathcal{F}_{n}\right)$.  We then clearly have that $\phi\left(n\right) =\card\left(\mathcal{F}_{n}^{*}\right)$ 
and that
$\psi\left(n\right)\sim n^{2}/\left(2\zeta\left(2\right)\right)$.
Next, observe that the statement in (\ref{fac:Farey}) implies that we have, for each continuous function $f:\left[0,1\right]\to\R_{\geq0}$,
  \[
\chi_{n}:=\frac{2\zeta\left(2\right)}{n^{2}}\sum_{r\in\mathcal{F}_{n}}f(r)\to \lambda(f), \hbox{ for $n$ tending
to infinity}.\]
 An application of Toeplitz's Lemma then gives that \[
\lim_{n \to \infty} \frac{1}{\log n}\sum_{k=1}^{n}\frac{1}{k}\chi_{k}= \lambda(f).\]
By  setting $f_{n}:=\sum_{p/n\in\mathcal{F}_{n}^{*}}f\left(p/n\right)$, 
we next observe that, for $n \geq 2$,  \[
\frac{1}{\log n}\sum_{k=1}^{n}\frac{1}{k}\chi_{k}= \frac{2\zeta\left(2\right)}{\log n}\sum_{k=1}^{n}\frac{1}{k^{3}}\sum_{m=1}^{k}f_{m}= \frac{2\zeta\left(2\right)}{\log n }\sum_{m=1}^{n}\sum_{k=m}^{n}\frac{1}{k^{3}}f_{m}.\]
 By comparing the sum $\sum_{k=m}^{n}k^{-3}$ with the corresponding
integral $\int_{m}^{n}x^{-3}\,\d x$, we obtain \begin{multline*}
\frac{\zeta\left(2\right)}{\log n}\sum_{q=1}^{n}\frac{f_{q}}{q^{2}}-\frac{\zeta\left(2\right)}{n^{2}\log n}\sum_{q=1}^{n}f_{q}\leq
\frac{1}{\log n }\sum_{k=1}^{n}\frac{1}{k}\chi_{k}\\
\leq\frac{\zeta\left(2\right)}{\log n}\sum_{q=1}^{n}\frac{f_{q}}{q^{2}}-\frac{\zeta\left(2\right)}{n^{2}\log n}\sum_{q=1}^{n}f_{q}+
\frac{\zeta\left(2\right)}{\log n }\sum_{q=1}^{n}\frac{f_{q}}{q^{3}}.\end{multline*}
 Finally, note that we clearly have that \[
\frac{\zeta\left(2\right)}{n^{2}\log n}\sum_{q=1}^{n}f_{q}\sim\frac{\lambda(f)}{2\log n}\]
 and  that
\[\frac{\zeta\left(2\right)}{\log n}\sum_{q=1}^{n}\frac{\phi\left(q\right)}{q^{3}}\underbrace{\frac{f_{q}}{\phi\left(q\right)}}_{\leq\left\Vert f\right\Vert _{\infty}}\leq\frac{\left\Vert f\right\Vert _{\infty}\left(\zeta\left(2\right)\right)^{2}}{\zeta\left(3\right)\log n}.\]
Hence, it now follows that \[
\lim_{n \to\infty}\frac{\zeta\left(2\right)}{\log n}\sum_{q=1}^{n}\frac{1}{q^{2}}\sum_{p/q\in\mathcal{F}_{q}^{*}}f\left(\frac{p}{q}\right)=
\lim_{n \to\infty}\frac{1}{\log n}\sum_{k=1}^{n}\frac{1}{k}\chi_{k}=\lambda(f) .\]
 This finishes the proof of the assertion  in Theorem \ref{Thm:atomic2}  (\ref{Farey}).
\end{proof}

\end{document}